\documentclass{amsart}

\usepackage{amsfonts}
\usepackage{mathrsfs}
\usepackage{cite}
\usepackage{graphicx}

\newcommand{\R}{{\mathbb R}}

\newtheorem{theo}{Theorem}[section]
\newtheorem{lemma}[theo]{Lemma}
\newtheorem{prop}[theo]{Proposition}

\newtheorem{remark}[theo]{Remark}
\newtheorem{defi}[theo]{Definition}
\newtheorem{example}[theo]{Example}

\title[Some applications of the Brascamp-Lieb inequality]{The Brascamp-Lieb inequality in Convex Geometry and in the Theory of Algorithms}
\author{K\'aroly J. B\"or\"oczky (R\'enyi Institute, Budapest)}

\begin{document}

\maketitle

\begin{abstract}
The Brascamp-Lieb inequality in harmonic analysis was proved by Brascamp and Lieb in the rank one case in 1976, and by Lieb in 1990. It says that in a certain inequality, the optimal constant can be determined by checking the inequality for centered Gaussian distributions. It was Keith M Ball's pioneering work around 1990 that led to various applications of the inequality in Convex Geometry, and even in Discrete Geometry, like Brazitikos' quantitative fractional version of the Helly Theorem. On the other hand, determining the optimal constant and possible Gaussian extremizers for the Brascamp-Lieb inequality can be formulated as a problem in terms of positive definite matrices, and this problem has intimate links to the Theory of Algorithms. 
\end{abstract}

\section{The Brascamp-Lieb-Barthe inequalities}
\label{secIntro}

For a proper linear subspace $E$ of $\R^n$ ($E\neq \R^n$ and $E\neq\{0\}$),
let $P_E$ denote the orthogonal projection into $E$.
We say that the  subspaces $E_1,\ldots,E_k$ of $\R^n$ and $p_1,\ldots,p_k>0$  form
a Geometric Brascamp-Lieb datum if they satisfy
\begin{equation}
\label{highdimcond0}
\sum_{i=1}^kp_iP_{E_i}=I_n.
\end{equation}
The name ``Geometric Brascamp-Lieb datum" coined by Bennett, Carbery, Christ, Tao \cite{BCCT08} comes from the following theorem, originating in the work of Brascamp, Lieb \cite{BrL76} and Ball \cite{Bal89,Bal91} in the rank one case
(${\rm dim}\,E_i=1$ for $i=1,\ldots,k$),
and Lieb \cite{Lie90} and Barthe \cite{Bar98} in the general case. In the rank one case, the Geometric Brascamp-Lieb datum is known by various names, like "John decomposition of the identity operator" (cf. Theorem~\ref{BrascampLiebRankOne} and Theorem~\ref{Johnmaxvol}), or tight frame, or Parseval frame in coding theory and computer science (see for example Casazza, Tran, Tremain \cite{CTT20}).

\begin{theo}[Brascamp-Lieb, Ball, Barthe]
\label{BLtheo}
For  the linear subspaces  $E_1,\ldots,E_k$ of $\R^n$ and $p_1,\ldots,p_k>0$ satisfying
\eqref{highdimcond0}, and for non-negative $f_i\in L_1(E_i)$, we have
\begin{equation}
\label{BL}
\int_{\R^n}\prod_{i=1}^kf_i(P_{E_i}x)^{p_i}\,dx
\leq \prod_{i=1}^k\left(\int_{E_i}f_i\right)^{p_i}
\end{equation}
\end{theo}
{\bf Remark} This is H\"older's inequality if $E_1=\ldots=E_k=\R^n$ and $P_{E_i}=I_n$, and hence $\sum_{i=1}^kp_i=1$.\\

We note that equality holds in Theorem~\ref{BLtheo}  if $f_i(x)=e^{-\pi\|x\|^2}$ for $i=1,\ldots,k$; and hence, each $f_i$ is a Gaussian density.
Actually,
Theorem~\ref{BLtheo} is an important special case discovered by Ball \cite{Bal91,Bal03}
in the rank one case and by Barthe \cite{Bar98} in the general case of the general Brascamp-Lieb inequality (cf. Theorem~\ref{BLgeneral}).

After partial results by Barthe \cite{Bar98}, Carlen,  Lieb,  Loss \cite{CLL04}
and Bennett, Carbery, Christ, Tao \cite{BCCT08}, it was Valdimarsson \cite{Val08}
who characterized equality in the Geometric Brascamp-Lieb inequality. In order to state his result, we need some notation.
Let $E_1,\ldots,E_k$  the proper linear subspaces of $\R^n$  and $p_1,\ldots,p_k>0$ satisfy
\eqref{highdimcond0}. As Bennett, Carbery, Christ, Tao \cite{BCCT08} observe, 
\eqref{highdimcond0} yields that for any non-zero linear subspace $V$, the map $\sum_{i=1}^k p_iP_V\circ P_{E_i}$ is the identity map on $V$, and hence considering traces show that
\begin{equation}
\label{sumEcapV}
\sum_{i=1}^k p_i\dim(E_i\cap V)\leq \dim V.
\end{equation}
In order to understand extremizers in \eqref{BL}, following
Carlen,  Lieb,  Loss \cite{CLL04} and Bennett, Carbery, Christ, Tao \cite{BCCT08}, we say that a non-zero linear subspace $V$
is a
critical subspace  if
$$
\sum_{i=1}^k p_i\dim(E_i\cap V)=\dim V,
$$
which is turn equivalent saying that
$$
\mbox{$E_i=(E_i\cap V)+ (E_i\cap V^\bot)$
for $i=1,\ldots,k$}
$$
by the argument leading to \eqref{sumEcapV} (cf. \cite{BCCT08}). We say that a critical subspace $V$ is indecomposable if
$V$ has no proper critical linear subspace.

Valdimarsson \cite{Val08} introduced the notions of independent subspaces and the dependent subspace.
We write $J$ to denote the set of $2^k$ functions
$\{1,\ldots,k\}\to\{0,1\}$. If $\varepsilon\in J$, then
let $F_{(\varepsilon)}=\cap_{i=1}^kE_i^{(\varepsilon(i))}$ where $E_i^{(0)}=E_i$ and $E_i^{(1)}=E_i^\bot$
for $i=1,\ldots,k$. We write $J_0$ to denote the subset of $\varepsilon\in J$ such that
${\rm dim}\,F_{(\varepsilon)}\geq 1$, and such an $F_{(\varepsilon)}$  is called independent following Valdimarsson \cite{Val08}. Readily $F_{(\varepsilon)}$ and $F_{(\tilde{\varepsilon})}$ are orthogonal
if $\varepsilon\neq\tilde{\varepsilon}$ for $\varepsilon,\tilde{\varepsilon}\in J_0$.
 In addition, we write $F_{\rm dep}$ to denote the orthogonal component of
$\oplus_{\varepsilon \in J_0}F_{(\varepsilon)}$. In particular, $\R^n$ can be written as a direct sum of pairwise orthogonal linear subspaces in the form
\begin{equation}
\label{independent-dependent0}
\R^n=\left(\oplus_{\varepsilon \in J_0}F_{(\varepsilon)}\right)\oplus F_{\rm dep}.
\end{equation}
Here it is possible that $J_0=\emptyset$, and hence $\R^n=F_{\rm dep}$, or
$F_{\rm dep}=\{0\}$, and hence $\R^n=\oplus_{\varepsilon \in J_0}F_{(\varepsilon)}$ in that case.

For a non-zero linear subspace $L\subset \R^n$, we say that a linear transformation $A:\,L\to L$ is positive definite
if  $\langle Ax,y\rangle=\langle x, Ay\rangle$ and $\langle x, Ax\rangle>0$ for any $x,y\in L\backslash\{0\}$.

\begin{theo}[Valdimarsson]
\label{BLtheoequa}
For  the proper linear subspaces  $E_1,\ldots,E_k$ of $\R^n$ and $p_1,\ldots,p_k>0$ satisfying
\eqref{highdimcond0}, let us assume that equality holds in the Brascamp-Lieb inequality \eqref{BL}
for non-negative $f_i\in L_1(E_i)$, $i=1,\ldots,k$.
If $F_{\rm dep}\neq\R^n$, then let $F_1,\ldots,F_\ell$ be the independent subspaces, and if
$F_{\rm dep}=\R^n$, then let $\ell=1$ and $F_1=\{0\}$. There exist
$b\in F_{\rm dep}$ and $\theta_i>0$ for $i=1,\ldots,k$,
integrable non-negative $h_{j}:\,F_j\to[0,\infty)$  for $j=1,\ldots,\ell$, and a positive definite matrix
$A:F_{\rm dep}\to F_{\rm dep}$ such that
the eigenspaces of $A$ are critical subspaces and
\begin{equation}
\label{BLtheoequaform}
f_i(x)=\theta_i e^{-\langle AP_{F_{\rm dep}}x,P_{F_{\rm dep}}x-b\rangle}\prod_{F_j\subset E_i}h_{j}(P_{F_j}(x))
\mbox{ \ \ \  for Lebesgue a.e. $x\in E_i$}.
\end{equation}
On the other hand, if for any $i=1,\ldots,k$, $f_i$ is of the form as in \eqref{BLtheoequaform}, then equality holds in \eqref{BL} for $f_1,\ldots,f_k$.
\end{theo}

Theorem~\ref{BLtheoequa} explains the term "independent subspaces" because the functions
$h_{j}$ on $F_j$ are chosen freely and independently  from each other.

A reverse form of the Geometric Brascamp-Lieb inequality was proved by
Barthe \cite{Bar98}. We write $\int^*_{\R^n}\varphi $ to denote the
outer integral for a possibly non-integrable function $\varphi:\,\R^n\to[0,\infty)$; namely, the infimum (actually minimum)
of $\int_{\R^n} \psi$ where $\psi\geq \varphi$ is Lebesgue measurable.

\begin{theo}[Barthe]
\label{RBLtheo}
For  the non-trivial linear subspaces  $E_1,\ldots,E_k$ of $\R^n$ and $p_1,\ldots,p_k>0$ satisfying
\eqref{highdimcond0}, and for non-negative $f_i\in L_1(E_i)$, we have
\begin{equation}
\label{RBL}
\int_{\R^n}^*\sup_{x=\sum_{i=1}^kp_ix_i,\, x_i\in E_i}\;\prod_{i=1}^kf_i(x_i)^{p_i}\,dx
\geq \prod_{i=1}^k\left(\int_{E_i}f_i\right)^{p_i}.
\end{equation}
\end{theo}
\noindent{\bf Remark.} This is the Pr\'ekopa-Leindler inequality (cf. Theorem~\ref{PL})
 if $E_1=\ldots=E_k=\R^n$ and $P_{E_i}=I_n$, and hence $\sum_{i=1}^kp_i=1$. \\

 We say that a function $h:\,\R^n\to[0,\infty)$
is log-concave if $h((1-\lambda)x+\lambda\,y)\geq h(x)^{1-\lambda}h(y)^\lambda$ for any
$x,y\in\R^n$ and $\lambda\in(0,1)$; or in other words, $h=e^{-W}$ for a convex function $W:\,\R^n\to(-\infty,\infty]$.
B\"or\"oczky, Kalantzopoulos, Xi \cite{BKX23} prove the following characterization of equality in the Geometric Barthe's inequality
\eqref{RBL}.

\begin{theo}[B\"or\"oczky, Kalantzopoulos, Xi]
\label{RBLtheoequa}
For linear subspaces  $E_1,\ldots,E_k$ of $\R^n$ and $p_1,\ldots,p_k>0$ satisfying
\eqref{highdimcond0},
if $F_{\rm dep}\neq\R^n$, then let $F_1,\ldots,F_\ell$ be the independent subspaces, and if
$F_{\rm dep}=\R^n$, then let $\ell=1$ and $F_1=\{0\}$.

If equality holds in the Geometric Barthe's inequality \eqref{RBL}
for non-negative $f_i\in L_1(E_i)$ with $\int_{E_i}f_i>0$, $i=1,\ldots,k$, then
\begin{equation}
\label{RBLtheoequaform}
f_i(x)=\theta_i e^{-\langle AP_{F_{\rm dep}}x,P_{F_{\rm dep}}x-b_i\rangle}\prod_{F_j\subset E_i}h_{j}(P_{F_j}(x-w_i)) \mbox{ \ \ \  for Lebesgue a.e. $x\in E_i$}
\end{equation}
where
\begin{itemize}
\item $\theta_i>0$,
$b_i\in E_i\cap F_{\rm dep}$ and $w_i\in E_i$ for $i=1,\ldots,k$,
\item $h_{j}\in L_1(F_j)$ is non-negative for $j=1,\ldots,\ell$, and in addition,
$h_j$ is log-concave if there exist $\alpha\neq \beta$ with $F_j\subset E_\alpha\cap E_\beta$,
\item $A:F_{\rm dep}\to F_{\rm dep}$ is a positive definite matrix such that
the eigenspaces of $A$ are critical subspaces.
\end{itemize}
On the other hand, if for any $i=1,\ldots,k$, $f_i$ is of the form as in \eqref{RBLtheoequaform} and equality holds for  all
$x\in E_i$ in \eqref{RBLtheoequaform}, then equality holds in \eqref{RBL} for $f_1,\ldots,f_k$.
\end{theo}

In particular, if for any $\alpha=1,\ldots,k$, the subspaces $\{E_i\}_{i\neq \alpha}$ span $\R^n$ in Theorem~\ref{RBLtheoequa},
then any extremizer of the Geometric Barthe's inequality is log-concave.

We note that Barthe's inequality \eqref{RBL} extends the celebrated Pr\'ekopa-Leindler inequality
Theorem~\ref{PL}
(proved in various forms by Pr\'ekopa \cite{Pre71,Pre73}, Leindler \cite{Lei72} and
 Borell \cite{Bor75}) whose equality case was clarified by Dubuc \cite{Dub77}
(see the survey Gardner \cite{gardner}). 
\begin{theo}[Pr\'ekopa, Leindler, Dubuc]
\label{PL}
For $m\geq 2$, $\lambda_1,\ldots,\lambda_m\in(0,1)$
with $\lambda_1+\ldots+\lambda_m=1$ and
 integrable $\varphi_1,\ldots,\varphi_m:\,\R^n\to[0,\infty)$, we have
\begin{equation}
\label{PLineq}
\int_{\R^n}^* \sup_{x=\sum_{i=1}^m\lambda_ix_i,\, x_i\in \R^n}\;\prod_{i=1}^m\varphi_i(x_i)^{\lambda_i}\,dx
\geq \prod_{i=1}^m\left(\int_{\R^n}\varphi_i\right)^{\lambda_i},
\end{equation}
and if equality holds and the left hand side is positive and finite, then there exist
a log-concave function $\varphi$ and
 $a_i>0$ and $b_i\in\R^n$ for $i=1,\ldots,m$ such that
$$
\varphi_i(x)=a_i\, \varphi(x-b_i)
$$
for Lebesgue a.e. $x\in\R^n$, $i=1,\ldots,m$.
\end{theo}

The explanation for the phenomenon concerning the log-concavity of $h_j$ in Theorem~\ref{RBLtheoequa} is as follows. Let
$\ell\geq 1$ and $j\in\{1,\ldots,\ell\}$, and hence $\sum_{E_i\supset F_j}p_i=1$. If $f_1,\ldots,f_k$ are of the form \eqref{RBLtheoequaform}, then equality in Barthe's inequality \eqref{RBL} yields
$$
\int^*_{F_j}\sup_{x=\sum_{E_i\supset F_j}p_i x_i\atop x_i\in F_j}h_{j}\Big(x_i-P_{F_j}w_i\Big)^{p_i}\,dx=
\prod_{E_i\supset F_j}\left(\int_{F_j}h_{j}\Big(x-P_{F_j}w_i\Big)\,dx\right)^{p_i}
\left(= \int_{F_j} h_j(x)\,dx\right).
$$
Therefore, if there exist $\alpha\neq \beta$ with $F_j\subset E_\alpha\cap E_\beta$, then
the equality conditions in the Pr\'ekopa-Leindler inequality
\eqref{PLineq} imply that $h_j$ is log-concave. On the other hand, if there exists $\alpha\in \{1,\ldots,k\}$ such that
$F_j\subset E_\beta^\bot$ for any  $\beta\neq\alpha$, then we do not have any condition on $h_j$, and $p_\alpha=1$.\\

For completeness, let us state and discuss the general Brascamp-Lieb inequality and its reverse form due to Barthe.
The following was proved by Brascamp, Lieb \cite{BrL76}  in the rank one case and Lieb \cite{Lie90} in general.

\begin{theo}[Brascamp-Lieb Inequality]
\label{BLgeneral}
Let $B_i:\R^n\to H_i$ be surjective linear maps where $H_i$ is $n_i$-dimensional Euclidean space,
$n_i\geq 1$, for $i=1,\ldots,k$ such that
$$
\cap_{i=1}^k {\rm ker}\,B_i=\{0\},
$$
and let $p_1,\ldots,p_k>0$ satisfy $\sum_{i=1}^kp_in_i=n$.
Then for non-negative $f_i\in L_1(H_i)$, we have
\begin{equation}
\label{BLgeneraleq}
\int_{\R^n}\prod_{i=1}^kf_i(B_ix)^{p_i}\,dx
\leq {\rm BL}(\mathbf{B},\mathbf{p})\cdot\prod_{i=1}^k\left(\int_{H_i}f_i\right)^{p_i}
\end{equation}
where  the optimal factor ${\rm BL}(\mathbf{B},\mathbf{p})\in(0,\infty]$ depending on $\mathbf{B}=(B_1,\ldots,B_k)$ and $\mathbf{p}=(p_1,\ldots,p_k)$ (which we call a  Brascamp-Lieb datum),
and ${\rm BL}(\mathbf{B},\mathbf{p})$ is determined by choosing centered Gaussians $f_i(x)=e^{-\langle A_ix,x\rangle}$ for some symmetric positive definite $n_i\times n_i$ matrix  $A_i$, $i=1,\ldots,k$ and $x\in H_i$.
\end{theo}
\noindent{\bf Remark} The Geometric Brascamp-Lieb Inequality is readily a special case of
\eqref{BLgeneraleq} where  ${\rm BL}(\mathbf{B},\mathbf{p})=1$. We note that \eqref{BLgeneraleq}
 is H\"older's inequality if $H_1=\ldots=H_k=\R^n$ and each $B_i=I_n$, and hence ${\rm BL}(\mathbf{B},\mathbf{p})=1$
and $\sum_{i=1}^kp_i=1$ in that case.

The condition $\sum_{i=1}^kp_in_i=n$ makes sure that for any $\lambda>0$, the inequality \eqref{BLgeneraleq} is invariant under replacing $f_1(x_1),\ldots,f_k(x_k)$ by $f_1(\lambda x_1),\ldots,f_k(\lambda x_k)$, $x_i\in H_i$.\\

We say that two Brascamp-Lieb datum $\{(B_i,p_i)\}_{i=1,\ldots,k}$ and
$\{(B'_i,p'_i)\}_{i=1,\ldots,k'}$ as in Theorem~\ref{BLgeneral}
are called equivalent if $k'=k$, $p'_i=p_i$, and there exists linear isomorphisms
$\Psi:\R^n\to\R^n$ and
 $\Phi_i:H_i\to H'_i$,
$i=1,\ldots,k$, such that $B'_i=\Phi_i\circ B_i\circ \Psi$. It was proved by
Carlen,  Lieb,  Loss \cite{CLL04} in the rank one case, and by Bennett, Carbery, Christ, Tao \cite{BCCT08}
in general that there exists a set of extremizers $f_1,\ldots,f_k$ for \eqref{BLgeneraleq} if and only
if the Brascamp-Lieb datum $\{(B_i,p_i)\}_{i=1,\ldots,k}$ is equivalent to some Geometric Brascamp-Lieb datum.
Therefore, Valdimarsson's Theorem~\ref{BLtheoequa} provides a full characterization of the equality case
in Theorem~\ref{BLgeneral}, as well.

The following reverse version of the Brascamp-Lieb inequality was proved by Barthe in \cite{Bar97} in the rank one case, and
in \cite{Bar98} in general.

\begin{theo}[Barthe's Inequality]
\label{RBLgeneral}
Let $B_i:\R^n\to H_i$ be surjective linear maps where $H_i$ is $n_i$-dimensional Euclidean space,
$n_i\geq 1$, for $i=1,\ldots,k$ such that
$$
\cap_{i=1}^k {\rm ker}\,B_i=\{0\},
$$
and let $p_1,\ldots,p_k>0$ satisfy $\sum_{i=1}^kp_in_i=n$.
Then for non-negative $f_i\in L_1(H_i)$, we have
\begin{equation}
\label{RBLgeneraleq}
\int_{\R^n}^*
\sup_{x=\sum_{i=1}^kp_i B_i^*x_i,\, x_i\in H_i}\;
\prod_{i=1}^kf_i(x_i)^{p_i}\,dx
\geq {\rm RBL}(\mathbf{B},\mathbf{p})\cdot \prod_{i=1}^k\left(\int_{H_i}f_i\right)^{p_i}
\end{equation}
where  the optimal factor ${\rm RBL}(\mathbf{B},\mathbf{p})\in[0,\infty)$ depends on the Brascamp-Lieb datum $\mathbf{B}=(B_1,\ldots,B_k)$ and $\mathbf{p}=(p_1,\ldots,p_k)$,
and ${\rm RBL}(\mathbf{B},\mathbf{p})$ is determined by choosing centered Gaussians $f_i(x)=e^{-\langle A_ix,x\rangle}$ for some symmetric positive definite $n_i\times n_i$ matrix  $A_i$, $i=1,\ldots,k$ and $x\in H_i$.
\end{theo}
\noindent{\bf Remark} The Geometric  Barthe's  Inequality is readily a special case of
\eqref{RBLgeneraleq}  where  ${\rm RBL}(\mathbf{B},\mathbf{p})=1$. We note that \eqref{RBLgeneraleq} is the Pr\'ekopa-Leindler inequality \eqref{PLineq} if $H_1=\ldots=H_k=\R^n$
 and each $B_i=I_n$, and hence ${\rm RBL}(\mathbf{B},\mathbf{p})=1$ and $\sum_{i=1}^kp_i=1$ in that case.

The condition $\sum_{i=1}^kp_in_i=n$ makes sure that for any $\lambda>0$, the inequality \eqref{RBLgeneraleq} is invariant under replacing $f_1(x_1),\ldots,f_k(x_k)$ by $f_1(\lambda x_1),\ldots,f_k(\lambda x_k)$, $x_i\in H_i$. \\

\begin{remark}[The relation between 
${\rm BL}(\mathbf{B},\mathbf{p})$ and ${\rm RBL}(\mathbf{B},\mathbf{p})$]
For a Brascamp-Lieb datum $\mathbf{B}=(B_1,\ldots,B_k)$ and $\mathbf{p}=(p_1,\ldots,p_k)$ as in 
Theorem~\ref{BLgeneral} and Theorem~\ref{RBLgeneral},
possibly ${\rm BL}(\mathbf{B},\mathbf{p})=\infty$ and  ${\rm RBL}(\mathbf{B},\mathbf{p})=0$ (see Section~\ref{secFiniteness} for the characterizastion when ${\rm BL}(\mathbf{B},\mathbf{p})$ and  ${\rm RBL}(\mathbf{B},\mathbf{p})$ are positive and finite). 

According to Barthe \cite{Bar98},  ${\rm BL}(\mathbf{B},\mathbf{p})<\infty$ if and only if  ${\rm RBL}(\mathbf{B},\mathbf{p})>0$, and in this case, we have
\begin{equation}
\label{BLRBL}
{\rm BL}(\mathbf{B},\mathbf{p})\cdot {\rm RBL}(\mathbf{B},\mathbf{p})=1.
\end{equation}
\end{remark}

Concerning extremals in Theorem~\ref{RBLgeneral},
Lehec \cite{Leh14} proved that if there exists some Gaussian extremizers for
Barthe's Inequality \eqref{RBLgeneraleq}, then the corresponding Brascamp-Lieb datum
 $\{(B_i,p_i)\}_{i=1,\ldots,k}$ is equivalent to some Geometric Brascamp-Lieb datum; therefore, the equality case of
\eqref{RBLgeneraleq} can be understood via Theorem~\ref{RBLtheoequa} in that case.

However, it is still not known whether having any extremizers in Barthe's  Inequality \eqref{RBLgeneraleq} yields the existence of Gaussian extremizers. One possible approach is to use iterated convolutions and
renormalizations as in Bennett, Carbery, Christ, Tao \cite{BCCT08} in the case of Brascamp-Lieb inequality.

The importance of the Brascamp-Lieb inequality is shown by the fact that besides harmonic analysis and convex geometry,
it has been also applied, for example, 
\begin{itemize}
\item in discrete geometry, like about a quantitative fractional Helly theorem by 
Brazitikos \cite{Bra14},
\item in combinatorics, like about exceptional sets by  
Gan \cite{Gan24},
\item in number theory, like the paper by Guo, Zhang \cite{GuZ19}, 
\item to get central limit theorems in probability, like the paper by Avram, Taqqu \cite{AvT06}.
\end{itemize}
We note the paper by Brazitikos \cite{Bra14} is especially interesting from the point of view that it does not simply consider the rank one Geometric Brascamp-Lieb inequality (cf. Theorem~\ref{BrascampLiebRankOne}) that is typically used for many inequalities in convex geometry, but an approximate version of it.

There are three main methods of proofs that work for proving both
the Brascamp-Lieb Inequality and its reverse form due to Barthe. The paper Barthe \cite{Bar98} used optimal transportation to prove
Barthe's Inequality (``the Reverse Brascamp-Lieb inequality") and reprove the Brascamp-Lieb Inequality simultaneously.
A heat equation argument was provided in the rank one case by
Carlen,  Lieb,  Loss \cite{CLL04} for the Brascamp-Lieb Inequality and by
Barthe, Cordero-Erausquin \cite{BaC04} for Barthe's  inequality. The general versions of both inequalities are proved via the heat equation approach by
Barthe, Huet \cite{BaH09}. Finally, simultaneous probabilistic arguments for the two inequalities are due to
Lehec \cite{Leh14}.
We note that
Chen,  Dafnis, Paouris \cite{CDP15}
and Courtade, Liu \cite{CoL21}, as well, deal systematically with finiteness conditions in Brascamp-Lieb and
 Barthe's  inequalities. 

Various versions of the Brascamp-Lieb inequality and its reverse form have been obtained by
Balogh, Kristaly \cite{BaK18}
Barthe \cite{Bar04}, Barthe, Cordero-Erausquin \cite{BaC04},
Barthe, Cordero-Erausquin,  Ledoux, Maurey \cite{BCLM11},
Barthe, Wolff \cite{BaW14,BaW22},
Bennett, Bez, Flock, Lee \cite{BBFL18},
 Bennett, Bez, Buschenhenke, Cowling, Flock \cite{BBBCF20}, Bennett, Tao \cite{BeT24},
Bobkov, Colesanti, Fragal\`a \cite{BCF14},
Bueno, Pivarov \cite{BuP21},
Chen,  Dafnis, Paouris \cite{CDP15},  Courtade, Liu \cite{CoL21},
Duncan \cite{Dun21}, Ghilli, Salani \cite{GhS17},
Kolesnikov, Milman \cite{KoM22},  Livshyts \cite{Liv21},
Lutwak, Yang, Zhang \cite{LYZ04,LYZ07},
Maldague \cite{Mal},  Marsiglietti \cite{Mar17},  Nakamura, Tsuji \cite{NaT},
Rossi, Salani \cite{RoS17,RoS19}.

\section{The Reverse Isoperimetric Inequality and the rank one Geometric Brascamp-Lieb inequality}

 For a compact convex set $K\subset\R^n$ with
${\rm dim}\,{\rm aff}\,K=m$, we write $|K|$ to denote the $m$-dimensional Lebesgue measure of $K$, and $S(K)$ to denote the surface area of $K$ in terms of the $(n-1)$-dimensional Hausdorff measure. In addition, let $B^n=\{x\in\R^n:\,\|x\|\leq 1\}$  be the Euclidean unit ball.\\

\noindent{\bf Remark.} For the box 
$X_\varepsilon=[-\varepsilon^{-(n-1)},\varepsilon^{-(n-1)}]\times [-\varepsilon,\varepsilon]^{n-1}$, we have
$|X_\varepsilon|=2^n$ but $S(X_\varepsilon)>1/\varepsilon$ (the area of a "long" facet); therefore,
the isoperimetric quotient $S(X_\varepsilon)^n/|X_\varepsilon|^{n-1}$ can be arbitrary large in general.
The "Reverse isoperimetric inequality" says that each convex body has a linear image whose isoperimetric quotient is at most as bad as of a regular simplex, and hence  "simplices have the worst isoperimetric quotient" up to linear transforms
(cf. Theorem~\ref{inverse-iso-simplex}). For origin symmetric convex bodies, "cubes have the worst isoperimetric quotient" up to linear transforms (cf. Theorem~\ref{inverse-iso-cube}).

Let $\Delta^n$ denote the regular simplex circumscribed  around $B^n$, and hence each facet touches $B^n$.

\begin{theo}[Reverse Isoperimetric Inequality, Keith Ball \cite{Bal91}]
\label{inverse-iso-simplex}
For any convex body $K$ in $\R^n$, there exists $\Phi\in {\rm GL}(n)$ such that
$$
\frac{S(\Phi K)^n}{|\Phi K|^{n-1}}\leq \frac{S(\Delta^n)^n}{|\Delta^n|^{n-1}}
=\frac{n^{3n/2}(n+1)^{(n+1)/2}}{n!},
$$ 
where strict inequality can be attained if and only if $K$ is not a simplex.
\end{theo}

We note that a {\it parallelepiped}\index{parallelepiped} is the linear image of a cube, and consider
the centered cube
$W^n=[-1,1]^n$  of edge length $2$.

\begin{theo}[Reverse Isoperimetric Inequality in the $o$-symmetric case, Keith Ball \cite{Bal89}]
\label{inverse-iso-cube}
For any $o$-symmetric convex body $K$ in $\R^n$, there exists $\Phi\in {\rm GL}(n)$ such that
$$
\frac{S(\Phi K)^n}{|\Phi K|^{n-1}}\leq \frac{S(W^n)^n}{|W^n|^{n-1}}=2^nn^n,
$$ 
where strict inequality can be attained if and only if $K$ is not a parallelepiped.
\end{theo}

We note that B\"or\"oczky, Hug \cite{BoH17b} and B\"or\"oczky, Fodor, Hug \cite{BFH19} prove stability versions 
Theorem~\ref{inverse-iso-simplex}
and Theorem~\ref{inverse-iso-cube}, respectively. 

To sketch the proof of the Reverse Isoperimetric Inequality Theorem~\ref{inverse-iso-simplex} and
Theorem~\ref{inverse-iso-cube} in order to show how it is connected to the Brascamp-Lieb inequality, we note that a polytope $P$ is circumscribed around $B^n$ if each facet of $P$ touches $B^n$. 

\begin{lemma}
\label{ballinbody}
If  $rB^n\subset K$ for a convex body $K$ in $\R^n$ and $r>0$, then $S(K)\leq \frac{n}r\,|K|$,
and equality holds if $K$ is a polytope circumscribed around $rB^n$.
\end{lemma}
\begin{proof} The inequality $S(K)\leq \frac{n}r\,|K|$ follows from 
$$
S(K)=\lim_{\varrho\to 0^+}\frac{|K+\varrho\,B^n|-|K|}{\varrho}\leq
\lim_{\varrho\to 0^+}\frac{|K+\frac{\varrho}r\,K|-|K|}{\varrho}=
\frac{n}r\,|K|.
$$
If $K$ is a polytope circumscribed around $rB^n$, then considering the bounded "cones" with apex $o$ and of height $r$ over the facets shows that
$|K|=\frac{r}n\,S(P)$ in this case.
\end{proof}

The proof of the Reverse Isoperimetric inequality both in the $o$-symmetric and non-symmetric cases
 is based on the rank one  Geometric Brascamp-Lieb inequality Theorem~\ref{BrascampLiebRankOne}.

\begin{theo}[Brascamp-Lieb, Keith Ball] 
\label{BrascampLiebRankOne}
If $u_1,\ldots,u_k\in S^{n-1}$ and $p_1,\ldots,p_k>0$ satisfy
\begin{equation}
\label{BLJohn0}
\sum_{i=1}^kp_i u_i\otimes u_i={\rm I}_n,
\end{equation}
and $f_1,\ldots,f_k\in L^1(\R)$ are non-negative, then
\begin{equation}
\label{BL0}
\int_{\R^n}\prod_{i=1}^kf_i(\langle x,u_i\rangle)^{p_i}\,dx\leq
\prod_{i=1}^k\left(\int_{\R}f_i\right)^{p_i}.
\end{equation}
\end{theo}
\noindent{\bf Remarks.}  
\begin{description}
\item[(i)] If $n=1$, then the  Brascamp-Lieb inequality (\ref{BL0}) is the H\"older inequality.

\item[(ii)] Inequality (\ref{BL0}) is optimal, and we provide two types of examples for equality:
\begin{itemize}
\item
If $u_1,\ldots,u_k\in S^{n-1}$ and $p_1,\ldots,p_k>0$ satisfy (\ref{BLJohn0}), and
$f_i(t)=e^{-\pi t^2}$  for $i=1,\ldots,k$, then each
$\int_{\R}f_i=1$, and
$$
\int_{\R^n}\prod_{i=1}^kf_i(\langle x,u_i\rangle)^{p_i}\,dx=
\int_{\R^n}e^{-\pi\sum_{i=1}^kp_i\langle x,u_i\rangle^2}\,dx=
\int_{\R^n}e^{-\pi\langle x,x\rangle^2}\,dx=1.
$$
\item If $u_1,\ldots,u_n$ is an orthonormal basis, $k=n$ and  $p_1=\ldots=p_n=1$, and hence (\ref{BLJohn0}) holds,
and $f_1,\ldots,f_n\in L^1(\R)$ any functions, then the Fubini Theorem yields 
$$ 
\int_{\R^n}\prod_{i=1}^nf_i(\langle x,u_i\rangle)^{p_i}\,dx=
\prod_{i=1}^n\left(\int_{\R}f_i\right)^{p_i}.
$$
\end{itemize}
\end{description}

More precisely, Theorem~\ref{BrascampLiebRankOne} is the so-called Geometric form of the
 rank one  Brascamp-Lieb inequality discovered by Keith Ball, which matches nicely the
form of John's theorem as in Theorem~\ref{Johnmaxvol} (see Keith Ball \cite{Bal92} or Gruber, Schuster \cite{GrS05} for the if and only if statement).

\begin{theo}[John] 
\label{Johnmaxvol}
For any convex $K\subset\R^n$, there exists a unique ellipsoid of maximal volume - the so-called John ellipsoid - contained in $K$.

Assuming that $B^n\subset K$, $B^n$ is the John ellipsoid of $K$ if and only if
   there exist
$u_1,\ldots,u_k\in S^{n-1}\cap \partial K$ and 
$p_1,\ldots,p_k>0$, $k\leq n(n+1)$, such that
\begin{align}
\label{John1}
\sum_{i=1}^kp_i u_i\otimes u_i&={\rm I}_n,\\
\label{John2}
\sum_{i=1}^kp_i u_i&=o
\end{align}
where ${\rm I}_n$ denotes the $n\times n$ identity matrix.

If $K$ is origin symmetric ($K=-K$), then we may assume that $k=2\ell$ for an integer $\ell\geq n$, and $p_{i+\ell}=p_i$ and $u_{i+\ell}=-u_i$ for $i\in\{1,\ldots,\ell\}$, and hence \eqref{John2} can be dropped.
\end{theo}
\noindent{\bf Remarks.} Assume that $B^n\subset K$ is the John ellipsoid of $K$ in Theorem~\ref{Johnmaxvol}.
\begin{itemize}
\item (\ref{John1})  yields that
$\langle x,y\rangle =\sum_{i=1}^kp_i\langle x,u_i\rangle\langle y,u_i\rangle$ for $x,y\in\R^n$, and hence the discrete measure $\mu$ on $S^{n-1}$ concentrated on $\{u_1,\ldots,u_k\}$
with $\mu(u_i)=p_i$ is called isotropic.

\item $\sum_{i=1}^k p_i=n$ follows by comparing traces in (\ref{John1}).

\item $\langle x,u_i\rangle\leq 1$ for $x\in K$ and $i=1,\ldots,k$ as $K$ and $B^n$ share the same supporting hyperplanes at 
$u_1,\ldots,u_k$.
\end{itemize}

Equality in Theorem~\ref{BrascampLiebRankOne} has been characterized by Barthe \cite{Bar98}. It is more involved; therefore, we only quote the special case that we need. 

\begin{theo}[Barthe]
\label{BLequa0}
Let  $\int_{\R}f_i>0$ for $i=1,\ldots,k$, such that none of the $f_i$s is Gaussian in Theorem~\ref{BrascampLiebRankOne},
and equality holds in (\ref{BL0}). Then there exists an orthonormal basis 
$e_1,\ldots,e_n$ of $\R^n$ such that $\{u_1,\ldots,u_k\}\subset\{\pm e_1,\ldots,\pm e_n\}$
and $\sum_{u_i\in\R e_p}p_i=1$ for each $e_p$,  and if  $u_i=-u_j$, then  
$f_i(t)=\lambda_{ij}f_j(-t)$ for $\lambda_{ij}>0$.
\end{theo}

It is a natural question how well an inscribed ellipsoid can approximate a convex body in terms of volume. This question was answered by Keith Ball \cite{Bal89,Bal91}, see Theorem~\ref{volume-ration-cube} for the origin symmetric case, and 
Theorem~\ref{volume-ratio-simplex} in general. 

\begin{theo}[Volume Ratio in the origin symmetric case, Keith Ball \cite{Bal89}]
\label{volume-ration-cube}
For any $o$-symmetric convex body $K$ in $\R^n$, the \index{volume ratio}maximal volume John ellipsoid $E\subset K$ satisfies
$$
\frac{|K|}{|E|}\leq \frac{|W^n|}{|B^n|}
=\frac{2^n}{\omega_n},
$$ 
where strict inequality is attained unless $K$ is a parallelepiped.
\end{theo}
\begin{proof} We may assume after a linear transformation that $E=B^n$. According to John's Theorem~\ref{Johnmaxvol}, 
 there exists  a symmetric set
$u_1,\ldots,u_{2\ell}\in S^{n-1}\cap \partial K$ and 
$p_1,\ldots,p_{2\ell}>0$ with $u_{i+\ell}=-u_i$ and $p_{i+\ell}=p_i$, $i=1,\ldots,\ell$, such that
$$
\sum_{i=1}^{2\ell}p_i u_i\otimes u_i={\rm I}_n.
$$
For $i=1,\ldots,2\ell$, let $f_i=\mathbf{1}_{[-1,1]}$. Now $K\subset P$ for 
the polytope $P=\{x\in\R^n:\,\langle x,u_i\rangle\leq 1$, $i=1,\ldots,2\ell\}$ according to the Remarks after 
John's Theorem~\ref{Johnmaxvol} where
$\mathbf{1}_P(x)=\prod_{i=1}^{2\ell}f_i(\langle x,u_i\rangle)=\prod_{i=1}^{2\ell}f_i(\langle x,u_i\rangle)^{p_i}$. 
It follows from the Brascamp-Lieb inequality (\ref{BL0})
and $\sum_{i=1}^{2\ell}p_i=n$ that
$$
|K|\leq |P|=\int_{\R^n}\prod_{i=1}^{2\ell}f_i(\langle x,u_i\rangle)^{p_i}\,dx\leq 
\prod_{i=1}^{2\ell}\left(\int_{\R}f_i\right)^{p_i}=2^{\sum_{i=1}^{2\ell}p_i}=2^n=|W^n|.
$$
If $|K|=|W^n|$, then $|K|=|P|$, and Theorem~\ref{BLequa0}  yields that $\ell=n$ and $u_1,\ldots,u_n$ is an orthonormal basis of $\R^n$; therefore, $K$ is a cube.
\end{proof}

Concerning the volume ratio of general convex bodies, we only sketch the argument because it involves a somewhat technical calculation.

\begin{theo}[Volume Ratio, Keith Ball \cite{Bal91}]
\label{volume-ratio-simplex}
For any convex body $K$ in $\R^n$, \index{volume ratio}the maximal volume John ellipsoid $E\subset K$ satisfies
$$
\frac{|K|}{|E|}\leq \frac{|\Delta^n|}{|B^n|}
=\frac{n^{n/2}(n+1)^{(n+1)/2}}{n!\omega_n},
$$ 
where strict inequality is attained unless $K$ is a simplex.
\end{theo}
\begin{proof}[Sketch of the proof of Theorem~\ref{volume-ratio-simplex}]
We may assume that $B^n$ is the John ellipsoid of $K$, and let $p_1,\ldots,p_k>0$ be the coefficients
and
$u_1,\ldots,u_k\in S^{n-1}\cap \partial K$ be the contact points satifying
\eqref{John1} and \eqref{John2} in John's Theorem~\ref{Johnmaxvol}; namely,
\begin{equation}
\label{John12VolumeRatio}
\sum_{i=1}^kp_i u_i\otimes u_i={\rm I}_n \mbox{ \ and \ }
\sum_{i=1}^kp_i u_i=o.
\end{equation}
Again, $K\subset P$ for 
the polytope $P=\{x\in\R^n:\,\langle x,u_i\rangle\leq 1$, $i=1,\ldots,k\}$ according to the Remarks after 
John's Theorem~\ref{Johnmaxvol}. 
The main idea is to lift $u_1,\ldots,u_k$ to $\R^{n+1}$, and employ the Brascamp-Lieb inequality in $\R^{n+1}$. In particular, $\R^n$ is identified with $w^\bot$ for a fixed
$w\in S^n\subset\R^{n+1}$, and let 
$\tilde{u}_i=-\sqrt{\frac{n}{n+1}}\cdot u_i+\sqrt{\frac{1}{n+1}}\cdot w$  and 
$\tilde{c}_i=\frac{n+1}{n}\cdot p_i$ for $i=1,\ldots,k$. Therefore,
$\sum_{i=1}^k\tilde{c}_i \tilde{u}_i\otimes \tilde{u}_i={\rm I}_{n+1}$ follows from 
\eqref{John12VolumeRatio}. For $i=1,\ldots,k$, we consider the probability density
$$
f_i(t)=\left\{
\begin{array}{rl}
e^{-t}&\mbox{if $t\geq 0$};\\
0&\mbox{if $t< 0$}
\end{array}
\right.
$$
on $\R$ where some not too complicated calculations show that
$$
\int_{\R^{n+1}}\prod_{i=1}^kf_i(\langle x,\tilde{u}_i\rangle)^{\tilde{c}_i}=\frac{|P|}{|\Delta^n|}.
$$
We conclude from the Brascamp-Lieb inequality \eqref{BL0} that $|K|\geq|P|\geq |\Delta^n|$.

If $|K|=|\Delta^n|$, then  $K=P$ and equality holds in the Brascamp-Lieb inequality. Therefore,
Theorem~\ref{BLequa0} provides an orthonormal basis $e_1,\ldots,e_{n+1}$ of
$\R^{n+1}$ such that $\{\tilde{u}_1,\ldots,\tilde{u}_k\}\subset\{\pm e_1,\ldots,\pm e_{n+1}\}$.
Since $\langle w,\tilde{u}_i\rangle=\sqrt{\frac{1}{n+1}}$ for $i=1,\ldots,k$, we conclude that
$k=n+1$ and $\tilde{u}_1,\ldots,\tilde{u}_{n+1}$ is an  an orthonormal basis of
$\R^{n+1}$, and hence $P$ is congruent to $\Delta^n$. 
\end{proof}

\begin{proof}[Proof of the Reverse Isoperimetric Inequality
Theorem~\ref{inverse-iso-simplex} and Theorem~\ref{inverse-iso-cube}:]
After applying an affine transformation, we may assume that the John ellipsoid of $K$ is $B^n$
both in Theorem~\ref{inverse-iso-simplex} and Theorem~\ref{inverse-iso-cube}.

For Theorem~\ref{inverse-iso-simplex},  Theorem~\ref{volume-ratio-simplex} yields that
$|K|\leq|\Delta^n|$, thus we deduce from
 Lemma~\ref{ballinbody} that 
$$
\frac{S(K)^n}{|K|^{n-1}}\leq \frac{n^n|K|^n}{|K|^{n-1}}=n^n|K|\leq n^n|\Delta^n|
=\frac{S(\Delta^n)^n}{|\Delta^n|^{n-1}}.
$$ 
If equality holds in Theorem~\ref{inverse-iso-simplex}, then the equality case of 
Theorem~\ref{volume-ratio-simplex} yields that $K$ is congruent to $\Delta^n$. 

For Theorem~\ref{inverse-iso-cube},  we use the same argument, only with 
Theorem~\ref{volume-ration-cube}  in place of Theorem~\ref{volume-ratio-simplex}.
\end{proof}

\section{The Loomis, Whitney inequality, the Bollobas-Thomason inequality and their dual forms}
\label{secBT}

In this section, we list some geometric inequalities that are direct consequences of the Geometric Brascamp-Lieb inequality \eqref{BL} and Barthe's 
 Geometric Reverse Brascamp-Lieb inequality \eqref{RBL}. 

We write $e_1,\ldots,e_n$ to denote an orthonomal basis of $\R^n$. 
The starting point is the classical Loomis-Whitney inequality
\cite{LoW49} from 1949 which follows from the Geometric Brascamp-Lieb inequality \eqref{BL} provided that $k=n$, $p_1=\ldots=p_n=\frac1{n-1}$ and $f_i$ is the characteristic function of $P_{e_i^\bot}K$.

\begin{theo}[Loomis, Whitney]
\label{Loomis-Whitney}
If $K\subset \R^n$ is compact and affinely spans $\R^n$, then
\begin{equation}
\label{Loomis-Whitney-ineq}
|K|^{n-1}\leq \prod_{i=1}^n|P_{e_i^\bot}K|,
\end{equation}
with equality if and only if
$K=\oplus_{i=1}^nK_i$ where ${\rm aff}K_i$ is a line parallel to $e_i$.
\end{theo}

Meyer \cite{Mey88} provided a dual form of the Loomis-Whitney inequality where equality holds for affine crosspolytopes
which follows from Barthe's Geometric Brascamp-Lieb inequality \eqref{RBL} provided that $k=n$, $p_1=\ldots=p_n=\frac1{n-1}$ and $f_i$ is the characteristic function of $e_i^\bot\cap K$.

\begin{theo}[Meyer]
\label{dual-Loomis-Whitney}
If $K\subset \R^n$ is compact convex with $o\in{\rm int}K$, then
\begin{equation}
\label{dual-Loomis-Whitney-ineq}
|K|^{n-1}\geq \frac{n!}{n^n}\prod_{i=1}^n|K\cap e_i^\bot|,
\end{equation}
with equality if and only if
$K={\rm conv}\{\pm\lambda_ie_i\}_{i=1}^n$ for $\lambda_i>0$, $i=1,\ldots,n$.
\end{theo}

We note that various Reverse and dual Loomis-Whitney type inequalities are proved by
Campi, Gardner, Gronchi \cite{CGG16}, Brazitikos {\it et al} \cite{BDG17,BGL18},
Alonso-Guti\'errez {\it et al} \cite{ABBC,AB}.

To consider a genarization of the Loomis-Whitney inequality and its dual form,
we set
$[n]:=\{1,\ldots,n\}$, and for a non-empty proper subset $\sigma\subset[n]$, we define
$E_\sigma={\rm lin}\{e_i\}_{i\in\sigma}$. For $s\geq 1$, we say that the not necessarily distinct proper
non-empty subsets $\sigma_1,\ldots,\sigma_k\subset[n]$ form an $s$-uniform cover of $[n]$ if each
$j\in[n]$ is contained in exactly $s$ of $\sigma_1,\ldots,\sigma_k$.

The Bollobas-Thomason inequality \cite{BoT95}  follows from the Geometric Brascamp-Lieb inequality \eqref{BL} provided that  $p_1=\ldots=p_k=\frac1{s}$ and $f_i$ is the characteristic function of $P_{E_{\sigma_i}}K$.

\begin{theo}[Bollobas, Thomason]
\label{Bollobas-Thomason}
If $K\subset \R^n$ is compact and affinely spans $\R^n$, and
$\sigma_1,\ldots,\sigma_k\subset[n]$ form an $s$-uniform cover of $[n]$ for $s\geq 1$, then
\begin{equation}
\label{Bollobas-Thomasson-ineq}
|K|^s\leq \prod_{i=1}^k|P_{E_{\sigma_i}}K|.
\end{equation}
\end{theo}

We note that the case when $k=n$, $s=n-1$, and hence when we may assume that $\sigma_i=[n]\backslash e_i$, is the
Loomis-Whitney inequality Therem~\ref{Loomis-Whitney}.

Liakopoulos \cite{Lia19} managed to prove a dual form of the Bollobas-Thomason inequality which
follows from Barthe's Geometric Brascamp-Lieb inequality \eqref{RBL} provided that $p_1=\ldots=p_n=\frac1{s}$ and $f_i$ is the characteristic function of $E_{\sigma_i}\cap K$.
For a finite set $\sigma$, we write $|\sigma|$ to denote its cardinality.

\begin{theo}[Liakopoulos]
\label{Liakopoulos}
If $K\subset \R^n$ is compact convex with $o\in{\rm int}K$, and
$\sigma_1,\ldots,\sigma_k\subset[n]$ form an $s$-uniform cover of $[n]$ for $s\geq 1$, then
\begin{equation}
\label{Liakopoulos-ineq}
|K|^s\geq \frac{\prod_{i=1}^k|\sigma_i|!}{(n!)^s}\cdot \prod_{i=1}^k|K\cap E_{\sigma_i}|.
\end{equation}
\end{theo}

The equality case of
the  Bollobas-Thomason inequality Theorem~\ref{Bollobas-Thomason} based on Valdimarsson \cite{Val08} has been known to the experts.
Let $s\geq 1$, and let $\sigma_1,\ldots,\sigma_k\subset[n]$ be an $s$-uniform cover
   of $[n]$. We say that
 $\tilde{\sigma}_1,\ldots,\tilde{\sigma}_l\subset[n]$
form a $1$-uniform cover of $[n]$ induced by
the $s$-uniform cover $\sigma_1,\ldots,\sigma_k$ if
$\{\tilde{\sigma}_1,\ldots,\tilde{\sigma}_l\}$ consists of all non-empty distinct subsets of $[n]$
of the form $\cap_{i=1}^k\sigma^{\varepsilon(i)}_i$ where $\varepsilon(i)\in\{0,1\}$ and $\sigma_i^0=\sigma_i$ and  $\sigma_i^1=[n]\setminus\sigma_i$. We observe that $\tilde{\sigma}_1,\ldots,\tilde{\sigma}_l\subset[n]$
actually form a $1$-uniform cover of $[n]$; namely, $\tilde{\sigma}_1,\ldots,\tilde{\sigma}_l$ is a partition of $[n]$.

\begin{theo}[Folklore]
\label{Bollobas-Thomason-eq}
Let $K\subset \R^n$ be compact and affinely span $\R^n$, and let
$\sigma_1,\ldots,\sigma_k\subset[n]$ form an $s$-uniform cover of $[n]$ for $s\geq 1$.
Then equality holds in \eqref{Bollobas-Thomasson-ineq} if and only if
$K=\oplus_{i=1}^l P_{E_{\tilde{\sigma}_i}}K$
where $\tilde{\sigma}_1,\ldots,\tilde{\sigma}_l$ is
the $1$-uniform cover  of $[n]$ induced by
 $\sigma_1,\ldots,\sigma_k$.
\end{theo}

On the other hand, Theorem~\ref{RBLtheoequa} yields the characterization
of the equality case of  the dual Bollobas-Thomason inequality Theorem~\ref{Liakopoulos} (cf. 
B\"or\"oczky, Kalantzopoulos, Xi \cite{BKX23}).

\begin{theo}
\label{Liakopoulos-eq}
Let $K\subset \R^n$ be compact convex  with $o\in{\rm int}K$, and let
$\sigma_1,\ldots,\sigma_k\subset[n]$ form an $s$-uniform cover of $[n]$ for $s\geq 1$.
Then equality holds in \eqref{Liakopoulos-ineq} if and only if
$K={\rm conv}\{K\cap F_{\tilde{\sigma}_i}\}_{i=1}^l$
where $\tilde{\sigma}_1,\ldots,\tilde{\sigma}_l$ is
the $1$-uniform cover  of $[n]$ induced by
 $\sigma_1,\ldots,\sigma_k$.
\end{theo}

\section{Finiteness of ${\rm BL}(\mathbf{B},\mathbf{p})$ and ${\rm RBL}(\mathbf{B},\mathbf{p})$}
\label{secFiniteness}

Let $\mathbf{B}=(B_1,\ldots,B_k)$ and $\mathbf{p}=(p_1,\ldots,p_k)$ be a Brascamp-Lieb datum  as in 
Theorem~\ref{BLgeneral} and Theorem~\ref{RBLgeneral}; namely, $B_i:\R^n\to H_i$ are surjective linear maps where $H_i$ is $n_i$-dimensional Euclidean space,
$n_i\geq 1$, for $i=1,\ldots,k$ such that
$$
\cap_{i=1}^k {\rm ker}\,B_i=\{0\},
$$
and $p_1,\ldots,p_k>0$ satisfy that $\sum_{i=1}^kp_in_i=n$.

The finiteness of the factor ${\rm BL}(\mathbf{B},\mathbf{p})$ in the Brascamp-Lieb inequality Theorem~\ref{BLgeneral} was characterized by Bennett, Carbery, Christ, Tao \cite{BCCT08}.

\begin{theo}[Bennett, Carbery, Christ, Tao \cite{BCCT08}, Barthe \cite{Bar98}] 
For  a Brascamp-Lieb datum $\mathbf{B}=(B_1,\ldots,B_k)$ and $\mathbf{p}=(p_1,\ldots,p_k)$ as above,
${\rm BL}(\mathbf{B},\mathbf{p})<\infty$ if and only if  ${\rm RBL}(\mathbf{B},\mathbf{p})>0$, which is in turn equivalent with the property that 
\begin{equation}
\label{BLfiniteV}
{\rm dim}\,V\leq \sum_{i=1}^k p_i\cdot  {\rm dim}\,(B_iV)
\end{equation}
for any  linear subspace $V\subset\R^n$. In this case, we have
\begin{equation}
\label{BLRBL0}
{\rm BL}(\mathbf{B},\mathbf{p})\cdot {\rm RBL}(\mathbf{B},\mathbf{p})=1.
\end{equation}
\end{theo}

Now fixing the surjective linear maps $B_i:\R^n\to H_i$, the question is a nice description of the set of all $\mathbf{p}=(p_1,\ldots,p_k)$ such that  ${\rm BL}(\mathbf{B},\mathbf{p})<\infty$. 
In addition, we say that $f_1,\ldots,f_k$ with positive integral are {\it extremizers} for the Brascamp-Lieb inequality
\eqref{BLgeneraleq} (Barthe's inequality \eqref{RBLgeneraleq}) if equality holds in \eqref{BLgeneraleq} (in \eqref{RBLgeneraleq}) for them. Moreover, $f_1,\ldots,f_k$ are called Gaussian extremizers if there exist some symmetric positive definite $n_i\times n_i$ matric  $A_i$ for $i=1,\ldots,k$ such that $f_i(x)=e^{-\langle A_ix,x\rangle}$ for $x\in H_i$.
According to Barthe \cite{Bar98},  $f_i(x)=e^{-\langle A_ix,x\rangle}$, $i=1,\ldots,k$, form a Gaussian extremizer for the Brascamp-Lieb inequality \eqref{BLgeneraleq} if and only if $f_i(x)=e^{-\langle A_i^{-1}x,x\rangle}$, $i=1,\ldots,k$, form a Gaussian extremizer for Barthe's Reverse Brascamp-Lieb  inequality \eqref{RBLgeneraleq}. 

\begin{theo}[Bennett, Carbery, Christ, Tao \cite{BCCT08}] 
\label{Brascamp-Lieb-polytope}
Fixing the surjective linear maps $B_i:\R^n\to H_i$ where $H_i$ is $n_i$-dimensional Euclidean space,
$n_i\geq 1$, for $i=1,\ldots,k$ such that
$$
\cap_{i=1}^k {\rm ker}\,B_i=\{0\},
$$
the set of all $\mathbf{p}=(p_1,\ldots,p_k)\in\R^k$ such that ${\rm BL}(\mathbf{B},\mathbf{p})<\infty$; namely,
$p_1,\ldots,p_k\geq 0$, $\sum_{i=1}^kp_in_i=n$  and \eqref{BLfiniteV} holds for any  linear subspace $V\subset\R^n$, is a $(k-1)$-dimensional bounded (closed) convex polytope $P_{\mathbf{B}}$. 

In addition, if $\mathbf{p}$ lies in the relative interior of this so-called 
 Brascamp-Lieb polytope $P_{\mathbf{B}}$ (strict inequality holds in \eqref{BLfiniteV} for any  linear subspace $V\subset\R^n$), then there exists a Gaussian extremizer both for the Brascamp-Lieb inequality \eqref{BLgeneraleq}and Barthe's inequality \eqref{RBLgeneraleq}.
\end{theo}
\noindent{\bf Remark.} For a Brascamp-Lieb datum $\mathbf{B}=(B_1,\ldots,B_k)$ and $\mathbf{p}=(p_1,\ldots,p_k)$, Bennett, Carbery, Christ, Tao \cite{BCCT08} proved that if there exists an extremizer for the Brascamp-Lieb inequality \eqref{BLgeneraleq}, then there exists a Gaussian extremizer, as well. However, the analogous statement is not known about  Barthe's Reverse Brascamp-Lieb  inequality \eqref{RBLgeneraleq}.\\

The fact that \eqref{BLfiniteV} for any  linear subspace $V\subset\R^n$ means only finitely many inequalities follows from the observation that there are only finitely many possible values of ${\rm dim}\,(B_iV)$ and ${\rm dim}\,V$. The vertices of $P_{\mathbf{B}}$ have been described by Barthe \cite{Bar98} in the rank one case (each $n_i=1$), and by   
 Valdimarsson \cite{Val10} in general.

According to Theorem~\ref{Brascamp-Lieb-polytope}, if $\mathbf{p}=(p_1,\ldots,p_k)$ lies in the relative interior of  $P_{\mathbf{B}}$, then there there exists Gaussian extremizers providing equality in the  Brascamp-Lieb inequality \eqref{BLgeneraleq}. However, if $\mathbf{p}$ lies in the relative boundary of  $P_{\mathbf{B}}$, then possibly we never have equality in the  Brascamp-Lieb inequality \eqref{BLgeneraleq}. On the other hand, we may have Gaussian extremizers for some other $\mathbf{p}$ lying in the relative boundary of  $P_{\mathbf{B}}$. We exhibit this phenomenon on the example of Young's  classical convolution inequality from. We recall that if
$f,g:\R\to[0,\infty)$ are measurable and  $p\geq 1$, then
\begin{align*}
\|f\|_p&=\left(\int_{\R}f^p\right)^{\frac1p};\\
 f* g(x)&=\int_{\R}f(y)g(x-y)\,dy.
\end{align*}

\begin{example}[Young's convolution inequality]
The original inequality by Young \cite{You12} from 1912
is of the following form:
If $p,q,s\geq 1$ with $\frac1p+\frac1q=\frac1s+1$, then there exists a minimal $c_{pq}>0$  such that for any measurable  $f,g:\R\to[0,\infty)$, we have
\begin{equation}
\label{Young1}
\|f*g\|_s\leq c_{pq}\cdot \|f\|_p\cdot \|g\|_q.
\end{equation}
Using the H\"older inequality and its equality case, we see that the version \eqref{Young1} of the Young inequality 
is equivalent with the following statement: Let $r\geq 1$ satisfy that $\frac1s+\frac1r=1$, and hence $\frac1p+\frac1q+\frac1r=2$. If $f,g,h:\R\to[0,\infty)$ are measurable, then
\begin{equation}
\label{Young2}
\int_{\R^2}f(y)g(x-y)h(x)\,dy\,dx\leq c_{pq}\cdot \|f\|_p\cdot \|g\|_q \cdot \|h\|_r.
\end{equation} 
Using the substitution 
$f_1=|f|^p$, $f_2=|g|^q$, $f_3=|h|^r$, $p_1=\frac1p$, $p_2=\frac1q$ and $p_3=\frac1r$, and hence $p_1+p_2+p_3=2$, \eqref{Young2} reads as
\begin{equation}
\label{Young3}
\int_{\R^2}f_1(y)^{p_1}f_2(x-y)^{p_2}f_3(x)^{p_3}\,dydx\leq c_{pq}\cdot \prod_{i=1}^3\left(\int_{\R}f_i\right)^{p_i}.
\end{equation}
Now \eqref{Young3}  is a proper Brascamp-Lieb inequality as in Theorem~\ref{BLgeneral} taking $H_i=\R$, $i=1,2,3$, $B_1(x,y)=y$, $B_2(x,y)=x-y$ and $B_3(x,y)=x$.

Let us see when ${\rm BL}(\mathbf{B},\mathbf{p})<\infty$ for the Brascamp-Lieb datum
$\mathbf{B}=(B_1,B_2,B_3)$ and $\mathbf{p}=(p_1,p_2,p_3)$.
Applying the condition \eqref{BLfiniteV} in the cases when the linear subspace $V$ has equation either $x=0$, or $x=y$, or $y=0$ yields the conditions $p_1+p_2\geq 1$, $p_1+p_3\geq 1$ and $p_2+p_3\geq 1$. Since $p_1+p_2+p_3=2$, we deduce that $P_{\mathbf{B}}\subset\R^3$ is a triangle with vertices $(1,1,0)$, $(1,0,1)$ and $(0,1,1)$. In turn, we also deduce that
$c_{pq}$ is finite in \eqref{Young1} if $p,q,s\geq 1$ satisfy $\frac1p+\frac1q=\frac1s+1$. Actualy, a simple argument based on the H\"older inequality yields that $c_{pq}\leq 1$.

Brascamp, Lieb \cite{BrL76}  proved that extremizers exists in \eqref{Young3} if and only if 
$\mathbf{p}=(p_1,p_2,p_3)$ lies either in the relative interior of $P_{\mathbf{B}}$, or 
$\mathbf{p}$ is a vertex of $P_{\mathbf{B}}$. In particular, if $\mathbf{p}$ lies on the relative interior of a side of $P_{\mathbf{B}}$, then no extremizers exist even if $c_{pq}$ is finite.

\end{example}

\section{Algorithmic and optimization aspects of the Brascamp-Lieb inequality}

Since for algorithms, we want to work with matrices and not with linear maps, we set $H_i=\R^{n_i}$ in the 
Brascamp-Lieb datum; therefore, for the whole section, $B_i:\R^n\to \R^{n_i}$ is a surjective linear map for $i=1,\ldots,k$, such that
\begin{equation}
\label{BiNonTrivial}
\cap_{i=1}^k {\rm ker}\,B_i=\{0\},
\end{equation}
$\mathbf{B}=(B_1,\ldots,B_k)$ and $\mathbf{p}=(p_1,\ldots,p_k)$ where $p_1,\ldots,p_k>0$ and $\sum_{i=1}^kp_in_i=n$.
Following Garg, Gurvits, Oliveira,  Wigderson \cite{GGOW18}, the main question we discuss in this section is how to determine effectively whether ${\rm BL}(\mathbf{B},\mathbf{p})$ is finite for a  Brascamp-Lieb datum $(\mathbf{B},\mathbf{p})$, and if finite, then how to approximate effectively its value.

We write $\mathcal{M}(m)$ to denote the set of  symmetric positive definite $m\times m$ matrices for $m\geq 1$. We note that if $A\in \mathcal{M}(m)$, then
$$
\int_{\R^m}e^{-\pi\langle Ax,x\rangle}\,dx=\sqrt{\det A}.
$$
It follows that the Brascamp-Lieb inequality Theorem~\ref{BLgeneral} proved by Lieb \cite{Lie90} is equivalent with the following statement.

\begin{theo}[Lieb \cite{Lie90}]
For  a Brascamp-Lieb datum $(\mathbf{B},\mathbf{p})$ as above, we have
\begin{equation}
\label{BLBpmatrices}
{\rm BL}(\mathbf{B},\mathbf{p})=\sup\left\{\sqrt{\frac{\prod_{i=1}^k(\det A_i)^{p_i}}{\det \sum_{i=1}^kp_iB_i^*A_iB_i}}:\,A_i\in\mathcal{M}(n_i),\; i=1,\ldots,k \right\}.
\end{equation}
\end{theo}

It follows from the condition  \eqref{BLfiniteV} on the subspaces $V$, that if we fix $\mathbf{p}$ in the Brascamp-Lieb datum $(\mathbf{B},\mathbf{p})$, then 
\begin{itemize}
\item the set of all $\mathbf{B}$ such that ${\rm BL}(\mathbf{B},\mathbf{p})<\infty$ is open (in the space of all possible $\mathbf{B}$).
\end{itemize}
Bennett, Bez,  Cowling, Flock \cite{BBCF17} prove the continuity of the  Brascamp-Lieb datum in terms of $\mathbf{B}$.

\begin{theo}[Bennett, Bez,  Cowling, Flock \cite{BBCF17}]
\label{BLBpBcont}
If we fix $\mathbf{p}$ in the Brascamp-Lieb datum $(\mathbf{B},\mathbf{p})$, then 
$\mathbf{B}\mapsto {\rm BL}(\mathbf{B},\mathbf{p})$ is a continuous function of $\mathbf{B}$,
including the values when ${\rm BL}(\mathbf{B},\mathbf{p})=\infty$.
\end{theo}

We say that the Brascamp-Lieb data $\mathbf{B}=(B_1,\ldots,B_k)$, $\mathbf{p}=(p_1,\ldots,p_k)$
and $\mathbf{B}'=(B'_1,\ldots,B'_m)$, $\mathbf{p}'=(p'_1,\ldots,p'_m)$ are equivalent, if $k=m$, $p'_i=p_i$ for $i=1,\ldots,k$, and there exist $\Phi\in{\rm GL}(n)$ and $\Psi_i\in{\rm GL}(n_i)$, $i=1,\ldots,k$, such that $B'_i=\Psi_i^{-1}B_i\Phi$, $i=1,\ldots,k$.

\begin{theo}[Bennett, Carbery, Christ, Tao \cite{BCCT08}] 
For the equivalent Brascamp-Lieb datums $(\mathbf{B},\mathbf{p})$ and $(\mathbf{B}',\mathbf{p}')$ as above, we have
\begin{equation}
\label{BLBpequivalence}
{\rm BL}(\mathbf{B}',\mathbf{p}')=\frac{\prod_{i=1}^k(\det \Psi_i)^{p_i}}{\det \Phi}
\cdot {\rm BL}(\mathbf{B},\mathbf{p}).
\end{equation}
\end{theo}

 Now the paper \cite{BCCT08} also showed that the existence of Gaussian maximizers is equivalent to saying that the Brascamp-Lieb datum is equivalent to a geometric one.
We write $B^*$ to denote the transpose of a matrix $B$.

\begin{defi}[Geometric Brascamp-Lieb datum]
\label{GeometricProperties}
We say that a  Brascamp-Lieb datum $(\mathbf{B},\mathbf{p})$ as at the beginning of the section is {\it geometric} if
\begin{description}
\item[Projection] $B_iB_i^*=I_{n_i}$ for $i=1,\ldots,k$;
\item[Isotropy] $\sum_{i=1}^kp_iB_i^*B_i=I_n$.
\end{description}
\end{defi}
\noindent{\bf Remarks.} 
\begin{itemize}
\item In this case, we can take $E_i=B_i^*\R^{n_i}$  and $B_i=P_{E_i}$ in order to obtain \eqref{highdimcond0}; namely, the equivalent the "Geometric Brascamp-Lieb datum" of Section~\ref{secIntro}.

\item In the geometric case, we have 
\begin{equation}
\label{geometricBLBp}
{\rm BL}(\mathbf{B},\mathbf{p})={\rm RBL}(\mathbf{B},\mathbf{p})=1 
\end{equation}
according to Keih Ball \cite{Bal89} in the rank one case, and Barthe \cite{Bar98} in general. One set of extremizers are
$f_i(x)=e^{-\pi\langle x,x\rangle}$ for $x\in\R^{n_i}$ and $i=1,\ldots,k$ (both for the Brascamp-Lieb inequality and Barthe's Reverse  Brascamp-Lieb inequality).

\item If both properties in Definition~\ref{GeometricProperties} hold, then the relations of non-triality (cf. \eqref{BiNonTrivial}) and $\sum_{i=1}^kp_in_i=n$ automatically hold.

\end{itemize}

If both the properties "Projection" and "Isotropic" hold, then we the Brascamp-Lieb constant is $1$ according \eqref{geometricBLBp}. However, already one of the conditions ensure that $1$ is a lower bound.

\begin{prop}[Garg, Gurvits, Oliveira,  Wigderson \cite{GGOW18}]
\label{ProjIsoBLBp}
If a  Brascamp-Lieb datum $(\mathbf{B},\mathbf{p})$ satisfies either the property "Projection" or "Isotropic"
in Definition~\ref{GeometricProperties}, then
\begin{equation}
\label{ProjIsoBLBp-eq}
{\rm BL}(\mathbf{B},\mathbf{p})\geq 1. 
\end{equation}
\end{prop}

Let us reformulate the results in the previous section by Bennett, Carbery, Christ, Tao \cite{BCCT08} about the finiteness of $(\mathbf{B},\mathbf{p})$ in the way such that it is used as a test for the algorithm by
 Garg, Gurvits, Oliveira,  Wigderson \cite{GGOW18}.

\begin{theo}[Bennett, Carbery, Christ, Tao \cite{BCCT08}] 
Let $(\mathbf{B},\mathbf{p})$ be a  Brascamp-Lieb datum  as at the beginning of the section.
\begin{description}
\item[${\rm BL}(\mathbf{B},\mathbf{p})$ finite] If $(\mathbf{B},\mathbf{p})$ is equivalent to a geometric Brascamp-Lieb datum, then $(\mathbf{B},\mathbf{p})$ is finite.

\item[${\rm BL}(\mathbf{B},\mathbf{p})$ infinite] If there exists a linear subspace $V\subset\R^n$ such that
$$
{\rm dim}\,V> \sum_{i=1}^k p_i\cdot  {\rm dim}\,(B_iV),
$$
 then $(\mathbf{B},\mathbf{p})$ is infinite.
\end{description}

\end{theo}
\noindent{\bf Remark.} Naturally, if $(\mathbf{B},\mathbf{p})$ is geometric, then  even there exists some maximizer
$A_1,\ldots,A_k$ in \eqref{BLBpmatrices} (and equivalently, some Gaussian maximizer in the Brascamp-Lieb inequality Theorem~\ref{BLgeneral}).\\

Theorem~\ref{BLBpBcont} about the continuity of the Brascamp-Lieb constant, and the above statements raise the hope that fixing $\mathbf{p}$ in the Brascamp-Lieb datum $(\mathbf{B},\mathbf{p})$ and varying $\mathbf{B}$, one might be able to find an efficient algorithm calculating ${\rm BL}(\mathbf{B},\mathbf{p})$. This was achieved by Garg, Gurvits, Oliveira,  Wigderson \cite{GGOW18}.

For any Brascamp-Lieb datum $(\mathbf{B},\mathbf{p})$ as at the beginning of the section, it can be easily achieved that at least one of the properties in Definition~\ref{GeometricProperties} hold.\\

\noindent{\bf "Projection-normalization":} For $C_i=B_iB_i^*$, $i=1,\ldots,k$, - that is an invertible $n_i\times n_i$ matrix by the non-triviality condition \eqref{BiNonTrivial} -, replace $B_i$ by $C_i^{-1/2}B'_i=B_i$, $i=1,\ldots,k$, and hence the Brascamp-Lieb datum $(\mathbf{B}',\mathbf{p})$ satisfies the "Projection" condition.\\

\noindent{\bf "Isotropy-normalization":} For $C=\sum_{i=1}^kp_iB_i^*B_i$ - that is an invertible $n\times n$ matrix by the non-triviality condition \eqref{BiNonTrivial} -, replace $B_i$ by $B'_i=B_iC^{-1/2}$, $i=1,\ldots,k$, and hence the Brascamp-Lieb datum $(\mathbf{B}',\mathbf{p})$ satisfies the "Isotropy" condition.\\

The key statement ensuring the effectiveness of the algorithm by Garg, Gurvits, Oliveira,  Wigderson \cite{GGOW18} is the following (we repeat Proposition~\ref{ProjIsoBLBp} in order to ensure the clarity of the statement).

\begin{theo}[Garg, Gurvits, Oliveira,  Wigderson \cite{GGOW18}]
Let $(\mathbf{B},\mathbf{p})$ be a  Brascamp-Lieb datum with ${\rm BL}(\mathbf{B},\mathbf{p})<\infty$
 where $\mathbf{B}$ has binary length
$b$ and  $,\mathbf{p}$ has common denominator $d$, and let $(\mathbf{B}',\mathbf{p})$ be the Brascamp-Lieb datum obtained from $(\mathbf{B},\mathbf{p})$ by either Projection-normalization or Isotropy-normalization.
\begin{description}
\item[Upper bound]  ${\rm BL}(\mathbf{B},\mathbf{p}) \leq \exp({\rm poly}(b, \log d))$.
\item[Lower bound] $(\mathbf{B}',\mathbf{p})\geq 1$.
\item[Progress per step]  If ${\rm BL}(\mathbf{B},\mathbf{p})>1+\varepsilon$ for $\varepsilon>0$, then
$$
{\rm BL}(\mathbf{B}',\mathbf{p})\leq \left(1-{\rm poly}\left(\frac{\varepsilon}{nd}\right)\right){\rm BL}(\mathbf{B},\mathbf{p}).
$$

\end{description}
\end{theo}

The basic idea of the algorithm by Garg, Gurvits, Oliveira,  Wigderson \cite{GGOW18} is that at the $m$th step, Projection-normalization is executed if $m$ is odd, and Isotropy-normalization is executed if $m$ is even.

\begin{theo}[Garg, Gurvits, Oliveira,  Wigderson \cite{GGOW18}]
There exists an algorithm such that on a  Brascamp-Lieb datum
$(\mathbf{B},\mathbf{p})$  where $\mathbf{B}$ has binary length
$b$ and  $,\mathbf{p}$ has common denominator $d$, and assuming an accuracy parameter $\varepsilon\in(0,1)$, the  algorithm
runs in time ${\rm poly}(b, d,1/\varepsilon)$, and 
\begin{itemize}
\item either computes a factor $(1+\varepsilon)$ approximation of ${\rm BL}(\mathbf{B},\mathbf{p})$
(in the case ${\rm BL}(\mathbf{B},\mathbf{p})<\infty)$,
\item or produces a linear subspace $V\subset\R^n$ satisfying the condition\\ ${\rm dim}\,V> \sum_{i=1}^k p_i\cdot  {\rm dim}\,(B_iV)$ (in the case  ${\rm BL}(\mathbf{B},\mathbf{p})=\infty$).
\end{itemize}
\end{theo}

Further properties of the Brascamp-Lieb datum$(\mathbf{B},\mathbf{p})$ when we fix $\mathbf{p}$ and vary $\mathbf{B}$ have been investigated by 
Bez, Gauvan, Tsuji \cite{BGT}.

\noindent{\bf Acknowledgement.} I am grateful to the two referees for all their helpful remarks improving the survey, and to J\'anos Pach for encouragement.

\end{document}